\documentclass[review]{elsarticle}

\usepackage{lineno,mathrsfs}
\modulolinenumbers[5]
\biboptions{numbers,sort&compress}

\journal{Journal of \LaTeX\ Templates}

\usepackage{url}
\usepackage{graphicx}
\usepackage{amsmath,amssymb,amsfonts,amsthm}
\usepackage{algorithmic}
\usepackage{textcomp}
\usepackage{xcolor}
\usepackage{color}

\usepackage{subfigure}
\allowdisplaybreaks[4]

\newtheorem{proposition}{Proposition}
\newtheorem{definition}{\bf Definition}
\newtheorem{remark}{Remark}
\newtheorem{property}{Property}
\begin{document}
	
	\begin{frontmatter}
		
		\title{Spectral graph fractional Fourier transform for directed graphs and its application}
		
		\author{Fang-Jia Yan, Bing-Zhao Li}
		\address{School of Mathematics and Statistics, Beijing Institute of Technology, Beijing, China, 100081}
		
		
		
		
		\begin{abstract}
In graph signal processing, many studies assume that the underlying network is undirected. Although the digraph model is rarely adopted, it is more appropriate for many applications, especially for real world networks. In this paper, we present a general framework for extending the graph signal processing to directed graphs in graph fractional domain. For this purpose, we consider a new definition for fractional Hermitian Laplacian matrix on directed graph and generalize the spectral graph fractional Fourier transform to directed graph (DGFRFT). Based on our new transform, we then define filtering, which is used in reducing unnecessary noise superimposed on temperature data. Finally, the performance of the proposed DGFRFT approach is also evaluated through numerical experiments using real-world directed graphs.
		\end{abstract}
		
		
	\end{frontmatter}
	
	
\section{Introduction}
\label{Intro}
In the field of network science and big data, it is necessary to expand the scope beyond classical time signal analysis and processing to accommodate the signals defined on the graph \cite{ramakrishna2021grid, chen2021graph, zhou2022novel, 4, 13}. The irregular structure of the underlying graph is different from the regular structure of classical signal processing, which brings great challenges to the analysis and processing of graph signal. Graph signal processing (GSP) offers effective tools to process such network data \cite{liao2022joint, gama2020graphs,morency2021graphon}. For instance, graph-supported signals can model vehicle trajectories over road networks \cite{deri2016new}. Research in GSP has only recently begun, but it is growing rapidly. The main contributions include wavelet and Fourier transforms \cite{2011Wavelets, Moura2013IEEE, 13, 25, 27}, sampling and reconstruction of graph signals \cite{kim2022quantization, yang2021efficient, chen2015discrete, romero2016kernel, yang2022reconstruction}, uncertainty principles \cite{tsitsvero2016signals, pasdeloup2019uncertainty}, filtering of graph signals \cite{xiao2021distributed, ozturk2021optimal}, etc. 

Different transforms of graph signal are still the core of GSP. Among them, the graph Fourier transform (GFT) acts as a cornerstone \cite{4,13}. In the literature, there are two frameworks for frequency analysis and processing of graph signals: (i) based on Laplacian matrix \cite{4}, (ii) based on adjacency matrix. The classical Laplacian based approach is limited to graph signals located on undirected graphs \cite{13}. In this method, frequency ordering is based on quadratic form and small eigenvalues correspond to low frequencies and vice versa. On the other hand, the adjacency based approach builds on the shift operator of the graph. This approach constructs the Fourier basis by using generalized eigenvectors of the adjacency matrix. Recently, some unique methods have been proposed to extend the GFT to directed graphs \cite{singh2016graph,sardellitti2017graph, sevi2018harmonic, shafipour2018directed, shafipour2019windowed, shafipour2018digraph, barrufet2021orthogonal,furutani2019graph, marques2020signal, seifert2021digraph, chen2022graph}. \cite{singh2016graph} uses the in-degree matrix and weight matrix to define the directed Laplacian. \cite{sardellitti2017graph} proposes an alternative approach that builds the graph Fourier basis as the set of orthonormal vectors that minimize a continuous extension of the graph cut size, known as the Lov{\'a}sz extension. \cite{sevi2018harmonic} introduces a novel harmonic analysis for functions defined on the vertices of a strongly connected directed graph. \cite{shafipour2018directed} studies the problem of constructing a GFT for directed graphs, which decomposes graph signals into different modes of variation with respect to the underlying network. \cite{shafipour2019windowed} proposes a methodology to carry out vertex-frequency analyses of graph signals on directed graph based on \cite{shafipour2018directed}. \cite{barrufet2021orthogonal} also defines new transform for directed graph, which using the Schur decomposition and leads to a series of embedded invariant subspaces for which orthogonal basis are available. \cite{furutani2019graph} extends graph signal processing to directed graphs based on the Hermitian Laplacian. \cite{marques2020signal} provides an overview of the current landscape of signal processing on directed graphs. Although these methods can construct Fourier bases with desirable properties, Fourier bases cannot completely retain information about the structure of the underlying graph.

In order to extract the local information of the graph signal, a new research direction has been proposed in graph signal processing (GSP), that is, fractional order \cite{25,26,27,28,ozturk2021optimal,yan2021windowed, ge2022optimal, kartal2022joint, yan2022multi}. The fractional order has gained considerable attention in the last 20 years in classical signal processing, and the application of fractional order to graphs has also aroused the interest of researchers \cite{ozturk2021optimal,yan2021windowed, ge2022optimal, kartal2022joint}. The graph fractional domain is a combination of the graph spectrum domain and fractional transform domain. The graph fractional Fourier transform (GFRFT) related to graph adjacency matrix is proposed in \cite{25}. Furthermore, a new spectral graph Fractional Fourier Transform (SGFRFT) related to graph Laplacian matrix is proposed in \cite{27}. GFRFT and SGFRFT show advantages in revealing the local characteristics of the graph signal. However, for directed graph, both of these transforms have their drawbacks. The Laplacian matrix for SGFRFT is constructed by an undirected graph, so SGFRFT does not apply to directed graphs. Although the GFRFT can use for directed graph, it has some potential problems. First, the basis comes from the Jordan decomposition is not orthonormal, and the Parseval’s identity does not hold and inner products are not preserved in the vertex domain and graph fractional domains. Second, numerical calculations of Jordan decomposition also often produce numerical instability even for medium-sized matrices \cite{golub1976ill}. 

The presence of directionality plays a crucial role when it comes to modeling social networks, technological networks,
biological and biomedical networks, as well as information networks \cite{newman2003structure, han2012extended, chui2018representation}.
For directed graphs, the existing studies in graph fractional domain are all based on adjacency matrix \cite{25,26}. In the continuous setting, fractional Fourier transform seeks the orthogonal bases which are eigenfunctions of the fractional Laplacian operator. The background naturally leads us to consider the eigenvectors of the fractional graph Laplacian operator in the discrete setting. Thus, we believe the Laplacian-based construction is more natural. In this paper, we aim to generalize SGFRFT to directed graphs. In the new transform, the fractional Laplacian operator is a simple extension of the Hermitian Laplacian discussed in \cite{zhang2021magnet} to directed graphs. Moreover, the new definition of SGFRFT links the existing Laplacian based approach to directed graph. The paper is organized as follows. We first review the GFT and SGFRFT as our foundation in Section \ref{Preliminary}. Section \ref{main} introduces the procedure about how to design spectral graph fractional Fourier transform for signal on directed graph. Then, an ideal filter and frequency selective filter are presented in Section \ref{filter}. For the last part, we present experiments on real directed graph, and an application of signal denoising using the filtering in the previous section.

\section{Preliminaries}
\label{Preliminary}

\subsection{Spectral graph theory}
An undirected weighted graph $\mathcal{G}=\{\mathcal{V}, \mathcal{E}, \mathcal{W}\}$ consists of a finite set of vertices $\mathcal{V}=\{v_0, \cdots, v_{N-1}\}$, where $N=|\mathcal{V}|$ is the number of vertices, a set of edges $\mathcal{E}=\{(i,j)|i,j\in\mathcal{V},j\sim i\}\subseteq\mathcal{V}\times\mathcal{V}$, and a weighted adjacency matrix $\mathcal{W}$. If the values of $\mathcal{W}$ are all in ${0, 1}$ then $\mathcal{W}$ is called an adjacency matrix \cite{chui2018representation}.  $\mathcal{W}=[{\mathcal{W}_{ij}}]\in\mathbb{R}^{N\times N}$ is defined as $\mathcal{W}_{ij}=w_{ij}$ if $(i,j)\in\mathcal{E}$ and $\mathcal{W}_{ij}=0$ otherwise. The non-normalized graph Laplacian is a symmetric difference operator $\mathcal{L}=D-\mathcal{W}$ \cite{13}, where $D:=diag(d_1, ..., d_N)$ is a diagonal degree matrix of $\mathcal{G}$, and $d_i:=\sum^{N}_{j=1}w_{ij}$. Let $\{\chi_0, \chi_1, \cdots, \chi_{N-1}\}$ be the set of orthonormal eigenvectors. Suppose that the corresponding Laplacian eigenvalues are sorted as $0=\lambda_0<\lambda_1\leq\lambda_2\leq\cdots\leq\lambda_{N-1}:=\lambda_{max}$.
Therefore
\begin{equation}
	\mathcal{L}=\mathbf{\chi}\mathbf{\Lambda} \mathbf{\chi}^H,
\end{equation}
where
\begin{equation}\label{tezheng}
	\mathbf{\chi} = 
		[\chi_0, \chi_1, \cdots, \chi_{N-1}],
\end{equation}
and the diagonal matrix is $\mathbf{\Lambda}=diag([\lambda_0, \lambda_1, \cdots, \lambda_{N-1}])$. The superscript $H$ represents the conjugate transpose of matrix.

The graph signal $f$ is defined as binding a scalar value to each vertex through the function $f:\mathcal{V}\rightarrow\mathbb{R}$.
Using the definition of (inverse) graph Fourier transform (GFT) as in \cite{13}, the GFT of $f$ is
\begin{equation}
	\widehat{f}(\ell)=\langle f,\chi_{\ell}\rangle=\sum^{N}_{n=1}f(n)\chi^{*}_{\ell}(n),  \ell=0,1,\cdots,N-1,
\end{equation}
where $*$ is complex conjugate.
The inverse GFT is given by
\begin{equation}\label{IGFT}
	\begin{split}
		f(n)=\sum^{N-1}_{\ell=0}\widehat{f}(\ell)\chi_{\ell}(n), n=0,1,\cdots,N-1.
	\end{split}
\end{equation}

\subsection{Spectral graph Fractional Fourier Transform}
The graph fractional Laplacian operator $\mathcal{L}_{\alpha}$ is defined by $\mathcal{L}_{\alpha}=\mathbf{\kappa}R\mathbf{\kappa}^{H}$, where $\alpha$ is the fractional order, $0<\alpha\leq1$,
$\ell=0,1,\cdots,N-1$ \cite{27}.
Note that
\begin{equation}\label{fractionallapalcian}
	\mathbf{\kappa} = \begin{bmatrix}
		\kappa_0,& \kappa_1,& \cdots,& \kappa_{N-1}
	\end{bmatrix}={\chi}^\alpha,
\end{equation}
and
\begin{equation}
	R=\text{diag}({\begin{bmatrix}
			r_0,& r_1,& \cdots,& r_{N-1}
	\end{bmatrix}})
	=\mathbf{\Lambda}^\alpha,
\end{equation}
so that
\begin{equation}
	r_\ell=\lambda_\ell^\alpha.
\end{equation}
In the follow-up part of this paper, computing the $\alpha$ power of a matrix always uses matrix power function.

The spectral graph Fractional Fourier Transform (SGFRFT) of any signal $f$ building on the graph $\mathcal{G}$ is defined by \cite{27}:
\begin{equation}
	\begin{split}
		\widehat{f}_{\alpha}(\ell)=\langle f,\kappa_{\ell}\rangle=\sum^{N}_{n=1}f(n)\kappa^{*}_{\ell}(n),  \ell=0,1,\cdots,N-1,
	\end{split}
\end{equation}
when $\alpha=1$, the SGFRFT degenerates into standard GFT.

The inverse SGFRFT is given by
\begin{equation}\label{IGFRFT}
	\begin{split}
		f(n)=\sum^{N-1}_{\ell=0}\widehat{f}_{\alpha}(\ell)\kappa_{\ell}(n), n=0,1,\cdots,N-1.
	\end{split}
\end{equation}

\begin{remark}
	In this paper, to distinguish directed graphs from undirected graphs,  we use $\mathcal{G}$ represents undirected graphs whereas $G$ represents directed graphs.
\end{remark}


%
%
%

\section{Spectral graph fractional Fourier transform for directed graph (DGFRFT)}
\label{main}

The orthogonality of eigenvectors gives the algebraic ideal properties of the SGFRFT, which makes the graph signal processing of undirected graphs develop well. However, since the eigenvectors of the fractional Laplacian operator of digraphs are usually not orthonormal, it is difficult to extend graph signal processing to digraphs simply. Our goal is to find a new way of defining fractional Laplacian matrix to keep the orthogonality of the eigenvectors and avoid the calculation of Jordan decomposition, i.e., Hermitian fractional Laplacian matrix. The Hermitian Laplacian is a complex matrix obtained
from an extension of the graph Laplacian \cite{zhang2021magnet, furutani2019graph, fanuel2018magnetic, f2020characterization}. It preserves the edge directionality and Hermitian property. Here we consider a directed graph $G=(V,E,W)$. $V$ is the set of $N$ vertices and $E$ is the set of directed edges. $W$ is the weight matrix of the graph, its element is defined as $w_{ij}$ which represents the weight of the directed edge from vertex $i$ to $j$. For directed graphs, the integers $d_i^{in}$ and $d_i^{out}$ specify the number of arrowheads directed toward and away from vertex $i$ respectively, \cite{hakimi1965degrees}. The in-degree of vertex $i$ is calculated as $d_i^{in}=\sum^N_{j=1}w_{ij}$, whereas, the out-degree is $d_i^{out}=\sum^N_{j=1}w_{ji}$. We define a new weight matrix as $W_s=[w_{ij,s}]$, where $w_{ij,s}=\frac{1}{2}(w_{ij}+w_{ji})$. $W_s=[w_{ij,s}]$ and $E_s$ (ignoring the directionality of $E$) uniquely
determine the corresponding undirected graph $G_s = (V, E_s, W_s)$ for directed $G$. The diagonal degree matrix $D_s$ of $G_s$ is $D(i,i):=\sum^{N}_{j=1}w_{ij,s}$. Then a Hermitian Laplacian of $G$ is defined as:
\begin{definition}(Hermitian graph Laplacian matrix)
\begin{equation}
	L=D_s-\Gamma_q\odot W_s,
\end{equation}
where $\odot$ is the hadamard product \cite{horn1990hadamard}, and $\Gamma_q$ is a Hermitian matrix which encodes the edge directionality of $G$.
\end{definition}
$\Gamma_q$ can take many forms, here's a simple example. Define $\Gamma_q=[\gamma(i,j)]$ that satisfies $\gamma(i,j)=\gamma(i,j)^H$.
$\gamma$ is a map from $V\times V$ to a unitary group of degree 1 and is written by \cite{f2020characterization}:
\begin{equation}
	\gamma_q{i,j}= e^{2\pi iq(w_{ij}-w_{ji})},
\end{equation}
when $w_{ij}=1$, $\gamma_q{i,j}= e^{2\pi iq}$ represents the edge from vertex $i$ to $j$ and $0\leq q<1$ is a rotation parameter. 
%
%

As $D_s$ and $W_s$ are real symmetric matrices, $L$ is Hermitian matrix. Let $v_k$ and $u_k$ be respectively the k-th eigenvalue and eigenvector of
the Hermitian Laplacian $L$. The eigendecomposition of $L$ can be written as:
\begin{equation}
	L=UVU^H,
\end{equation}
where $U=[u_0, u_1, \cdots, u_{N-1}]$. $V=diag([v_0, v_1, \cdots, v_{N-1}])$ is a real diagonal matrix, and Hermitian Laplacian eigenvalues are sorted as $0\leq v_0<v_1\leq v_2\leq\cdots\leq v_{N-1}:=v_{max}$.

The graph Hermitian fractional Laplacian matrix for directed graph is given by $L_{\alpha,d}=PQP^H$, in which
\begin{equation}\label{fractionallapalcian}
	P = \begin{bmatrix}
		p_0,& p_1,& \cdots,& p_{N-1}
	\end{bmatrix}={U}^\alpha,
\end{equation}
and
\begin{equation}
	Q=\text{diag}({\begin{bmatrix}
			\xi_0,& \xi_1,& \cdots,& \xi_{N-1}
	\end{bmatrix}})
	=V^\alpha,
\end{equation}
that is
\begin{equation}
	\xi_\ell=v_\ell^\alpha, \quad l=0,1,\cdots,N-1.
\end{equation}

We select two different $\alpha$ as examples to show the spectral properties of Hermitian fractional Laplacian matrix $L_{\alpha,d}$ in Fig. \ref{eigenvaluefigure}.  When the rotation parameter $q=0$, the Hermitian fractional Laplacian matrix degrades to normal fractional Laplacian matrix as each element of $\Gamma_q$ equals to $1$. Therefore, when the values of $q$ are relatively small, the spectrum of the Hermitian fractional Laplacian is similar to the spectrum of the fractional Laplacian, otherwise, when $q$ is large, there will be oscillation. 

\begin{figure}
	\centering
	\subfigure[Eigenvalues of Hermitian fractional Laplacian when $\alpha=0.8$]{
		\label{alpha=0.8}
		\includegraphics[width=0.45\textwidth]{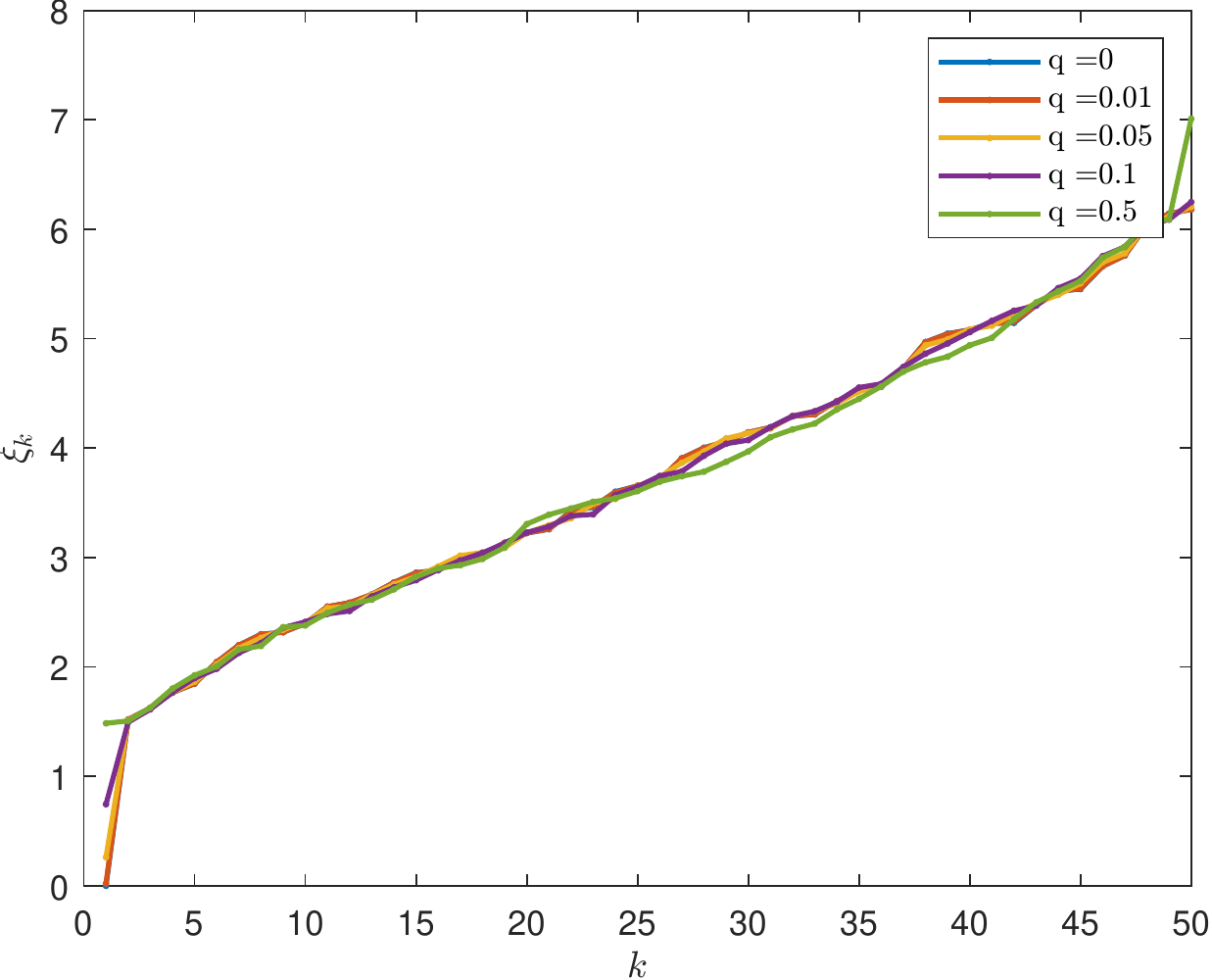}}
	\subfigure[Eigenvalues of Hermitian fractional Laplacian when $\alpha=0.6$]{
		\label{alpha=0.6}
		\includegraphics[width=0.45\textwidth]{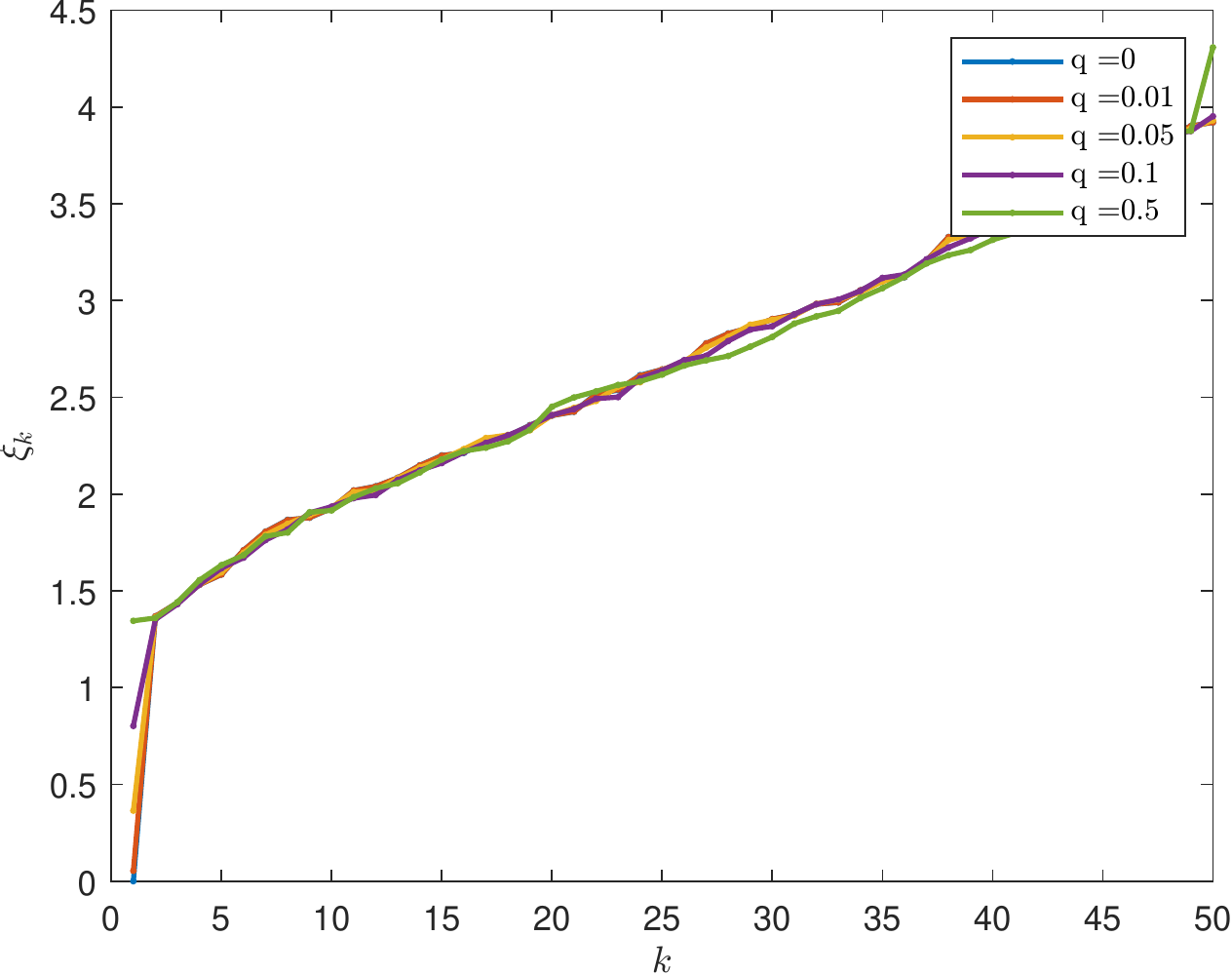}}
	\caption{Eigenvalues of Hermitian fractional Laplacian with respect to different fractional orders.}
	\label{eigenvaluefigure}
\end{figure}

The total variation for graph signals is defined as an absolute sum of the
discrete difference of a graph signal \cite{ono2015total}. It was first introduced in as an extension of the original total variation (TV) \cite{chan2001digital}. We define the total variation of a graph signal $f$ as
\begin{equation}
\operatorname{TV}(\boldsymbol{f}):=\sum_{(i, j) \in \mathcal{E}}|f(i)-f(j)|^2.
\end{equation}
to measure the smoothness of eigenvectors of $L_{\alpha,d}$. The total variation has an intuitive interpretation: it compares how the signal varies with time or space and calculates a cumulative magnitude of the signal. The smaller the difference between the original signal $f(i)$ and $f(j)$, the lower the signal's variation. Fig. \ref{TV} shows total variations of eigenvectors of $L_{\alpha,d}$
on a random directed graph with 50 nodes when the fractional parameter $\alpha=0.8$ and $0.6$. To construct a random directed graph, we fix $n$ nodes and for each pair of nodes we generate a direct edge with probability $p$ ($n=50, p=0.1$).

\begin{figure}[t]
	\centering
	\subfigure[Total variations of eigenvectors of Hermitian fractional Laplacian when $\alpha=0.8$]{
		\label{TValpha=0.8}
		\includegraphics[width=0.45\textwidth]{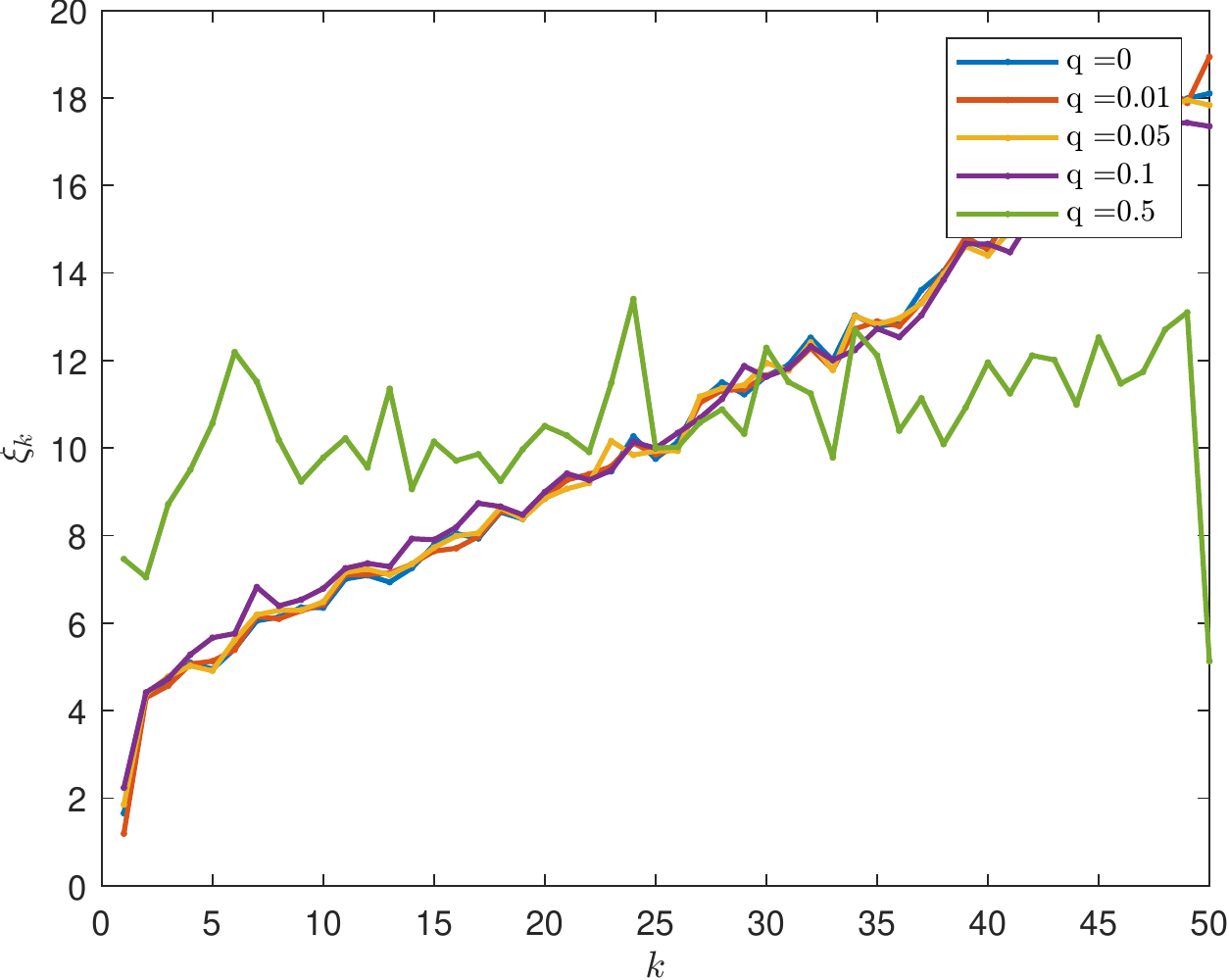}}
	\subfigure[Total variations of eigenvectors of Hermitian fractional Laplacian when $\alpha=0.6$]{
		\label{TValpha=0.6}
		\includegraphics[width=0.45\textwidth]{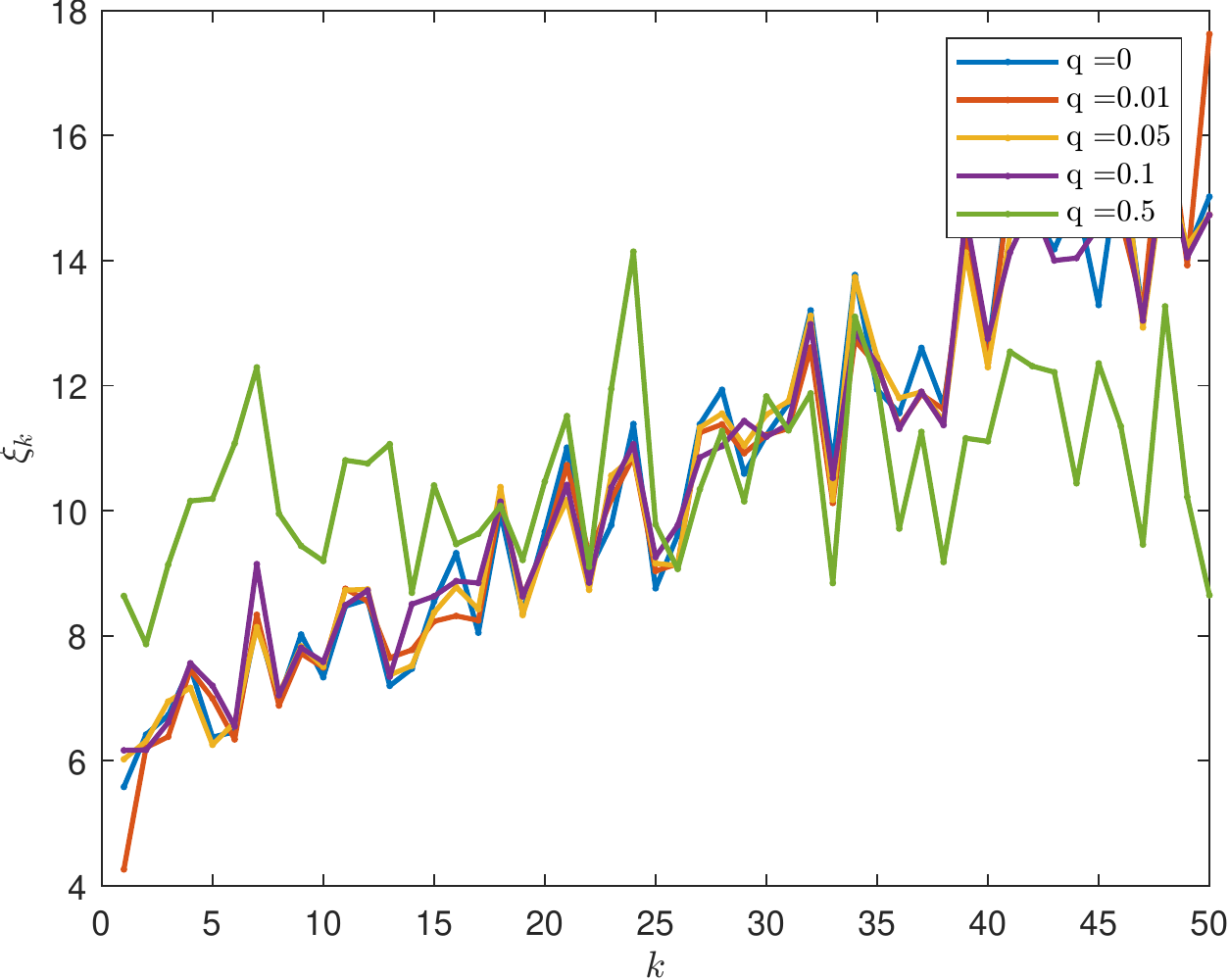}}
	\caption{Total variations of eigenvectors of Hermitian fractional Laplacian with respect to different fractional orders.}
	\label{TV}
\end{figure}

Then we prove that $L_{\alpha,d}$ is a positive semi-definite Hermitian matrix.
\begin{proposition}\label{Hermitian}
	For any fractional order $\alpha$, the $L_{\alpha,d}$ is a Hermitian matrix:
	\begin{equation}
		(L_{\alpha,d})^{H}=L_{\alpha,d}.
	\end{equation}
\end{proposition}
\begin{proof}
	Since $V$ is a real diagonal matrix, $(V^\alpha)^H=V^\alpha$.
	\begin{equation}
	   (L_{\alpha,d})^{H}=(PQP^H)^H=(U^\alpha V^\alpha(U^\alpha)^H)^H=U^\alpha V^\alpha(U^\alpha)^H=PQP^H=L_{\alpha,d}.
	\end{equation}
\end{proof}

\begin{proposition}
	For any fractional order $\alpha$, the $L_{\alpha,d}$ is a positive semi-definite matrix.
\end{proposition}
\begin{proof}
	From Proposition \ref{Hermitian}, we know $V^\alpha$ is real diagonal matrix. Let $J=V^{\alpha/2}(U^\alpha)^H$, so $L_{\alpha,d}=J^HJ$. For any signal $f\in\mathbb{C}^N$, note that
	\begin{equation}
		f^HL_{\alpha,d}f=f^HU^\alpha V^\alpha(U^\alpha)^Hf=f^HJ^HJf=(Jf)^HJf\geq0
	\end{equation}
\end{proof}

Clearly, in this case, this Hermitian fractional Laplacian matrix has a set of orthonormal eigenvectors. This orthogonality allows the basic concepts of graph signal processing of undirected graphs to be extended directly to those of directed graphs.

\begin{definition}(Directed graph fractional Fourier transform)
The spectral graph Fractional Fourier Transform for directed graph (DGFRFT) of any signal $f$ building on the graph $G$ is defined as:
\begin{equation}
	\begin{split}
		\widehat{f}_{\alpha,d}(\ell)=\langle f,p_{\ell}\rangle=\sum^{N}_{n=1}f(n)p^{*}_{\ell}(n), \ell=0,1,\cdots,N-1.
	\end{split}
\end{equation}
By the matrix form, the DGFRFT is
\begin{equation}
		\widehat{f}_{\alpha,d}=P^Hf,
\end{equation}
when $\alpha=1$, the DGFRFT degenerates into GFT for directed graph. 

The inverse DGFRFT is given by
\begin{equation}\label{IGFRFT}
	\begin{split}
		f(n)=\sum^{N-1}_{\ell=0}\widehat{f}_{\alpha,d}(\ell)p_{\ell}(n), n=0,1,\cdots,N-1.
	\end{split}
\end{equation}
\end{definition}

Now we have the definition of DGFRFT, it has some useful properties.
\begin{property} Unitarity: $P^{-1}=(U^\alpha)^{-1}=U^{-\alpha}=P^H$.
\end{property}
\begin{property} Index additivity: $U^{-\alpha}\circ U^{-\beta}=U^{-\beta}\circ U^{-\alpha}=U^{-(\alpha+\beta)} $.
\end{property}
\begin{property} Reduction to SGFRFT when $w_{ij}=w_{ji}$.
\end{property}
\begin{property} Parseval relation holds, for any signal $f$ and $g$ defined on the directed graph $G$ we have:
\begin{equation}
	\begin{split}
		\langle f,g\rangle=\langle \widehat{f}_{\alpha,d},\widehat{g}_{\alpha,d}\rangle.
	\end{split}
\end{equation}
If $f=g$, then
\begin{equation}\label{parseval}
	\begin{split}
		&\sum^{N}_{n=1}|f(n)|^{2}=\|f\|_{2}^{2}=\langle f,f\rangle\\
		&=\langle \widehat{f}_{\alpha,d},\widehat{f}_{\alpha,d}\rangle=\|\widehat{f}_{\alpha,d}\|_{2}^{2}=\sum^{N-1}_{\ell=0}|\widehat{f}_{\alpha,d}(\ell)|^{2}.
	\end{split}
\end{equation}
\end{property}

DGFRFT is a new transform focused on directed graph in graph fractional domain. Compared with GFRFT, it has several significant advantages.
First, the columns of U are linearly independent eigenvectors, and at the same time they are orthogonal. This results in DGFRFT that preserves the inner product when passed from the vertex domain to the graph fractional domain. In addition, DGFRFT preserves edge directionality. Finally, computing the GFRFT needs to do Jordan decomposition, when the size of the graph exceeds the median value, the calculation of Jordan decomposition will lead to serious and difficult numerical instability. As our new Laplacian matrix for directed graph is Hermitian matrix, the Jordan decomposition can be avoided in the calculation of DGFRFT.

\section{Directed graph filtering}
\label{filter}
\subsection{Spectral graph filtering}
In classical signal processing, filtering can be defined by convolution. The convolution of signal a with b is the result of signal a filtered by b. Therefore, to define the directed graph filtering in graph fractional domain, first we need to define the convolution operator. Convolution in the time domain is equivalent to multiplication in the Fourier domain. For directed graph, the graph fractional convolution operator is defined as the following form consistent with the classical convolution by using the directed graph fractional Laplacian eigenvector.
\begin{definition}(Convolution operator)
For any graph signal $f$ and $g$ which underlying graph structure is directed, their graph fractional convolution $*$ is
	\begin{equation}\label{convolution}
		(f*g)(n)=\sum_{l=0}^{N-1}	\widehat{f}_{\alpha,d}(\ell)	\widehat{g}_{\alpha,d}(\ell)p_{\ell}(n).
	\end{equation}
\end{definition}

A graph filter is a system which takes a graph signal as input and produces another graph signal as output \cite{chen2014signal}. Given a input directed graph signal $f_in$, the filtering is defined by the convolution of $f_in$ and a filter $h$. Thus the spectral directed graph filtering in vertex domain is:
\begin{equation}
		f^{out}(n)=(f*h)(n)=\sum^{N-1}_{\ell=0}\widehat{f}_{\alpha,d}^{in}(\ell)\widehat{h}_{\alpha,d}(\ell)p_{\ell}(n), n=0,1,\cdots,N-1,
\end{equation}
and in graph fractional domain:
\begin{equation}
		\widehat{f}_{\alpha,d}^{out}(\ell)=\widehat{f}_{\alpha,d}^{in}(\ell)\widehat{h}_{\alpha,d}(\ell), \ell=0,1,\cdots,N-1.
\end{equation}

\subsection{Frequency selective filtering}
The above spectral graph filtering is an ideal filter. It retains all frequencies within a given range and completely removes those that are out of range. Because the graph has a finite number of points, it is impossible to design such a filter. A frequency selective filter is a system that isolates the specific frequency components and excludes all others. It's easier to achieve in application. The ideal response of the frequency selective filter is \cite{iacobucci2005frequency}
\begin{equation}
	h_d(l)=\left\{\begin{array}{l}
		1,l\in\text{passband}\\
		0,l\in\text{stopband}.
	\end{array}\right.
\end{equation}
Lowpass, bandpass and highpass filters are three common frequency selective filters.

Using $h_d(l)$, the graph signal $f$ is filtered in graph fractional domain by:
\begin{equation}
		\widehat{f}_{\alpha,d}^{out}=J_d\widehat{f}_{\alpha,d},
\end{equation}
where $J_d=diag([h_d(\xi_1), h_d(\xi_2), \cdots, h_d(\xi_N)])$, and $P^{-1}f=\widehat{f}_{\alpha,d}$ is the DGFRFT of $f$.

Or, equivalently, vertex domain filtering can be obtained by inverse DGFRFT:
\begin{equation}
	f^{out}=P\widehat{f}_{\alpha,d}^{out}=PJ_d\widehat{f}_{\alpha,d}=PJ_dP^{-1}f=Hf.
\end{equation}
where $H=PJ_dP^{-1}$ represents the transfer matrix.

\section{Application}
\begin{figure}[t]
	\centering
	\includegraphics[width=0.7\textwidth]{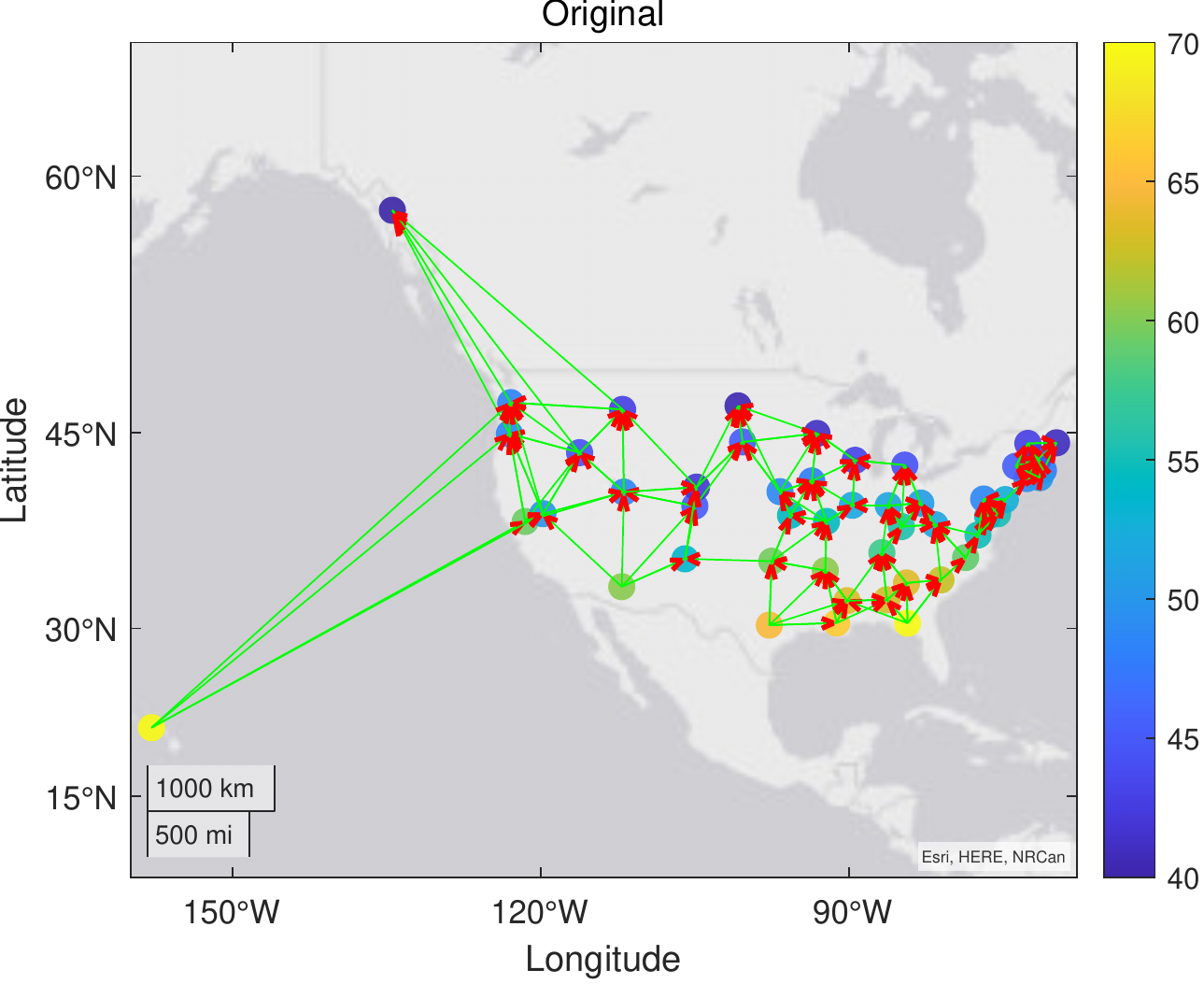}
	\caption{Average temperature data measured by 50 meteorological state stations in the US.}
	\label{USTEM}
\end{figure}

\begin{figure}[t]
	\centering
	\subfigure[original signal.]{
		\label{direct_original}
		\includegraphics[width=0.4\textwidth]{direct_original.pdf}}
	\subfigure[noisy signal.]{
		\label{direct_noisy}
		\includegraphics[width=0.4\textwidth]{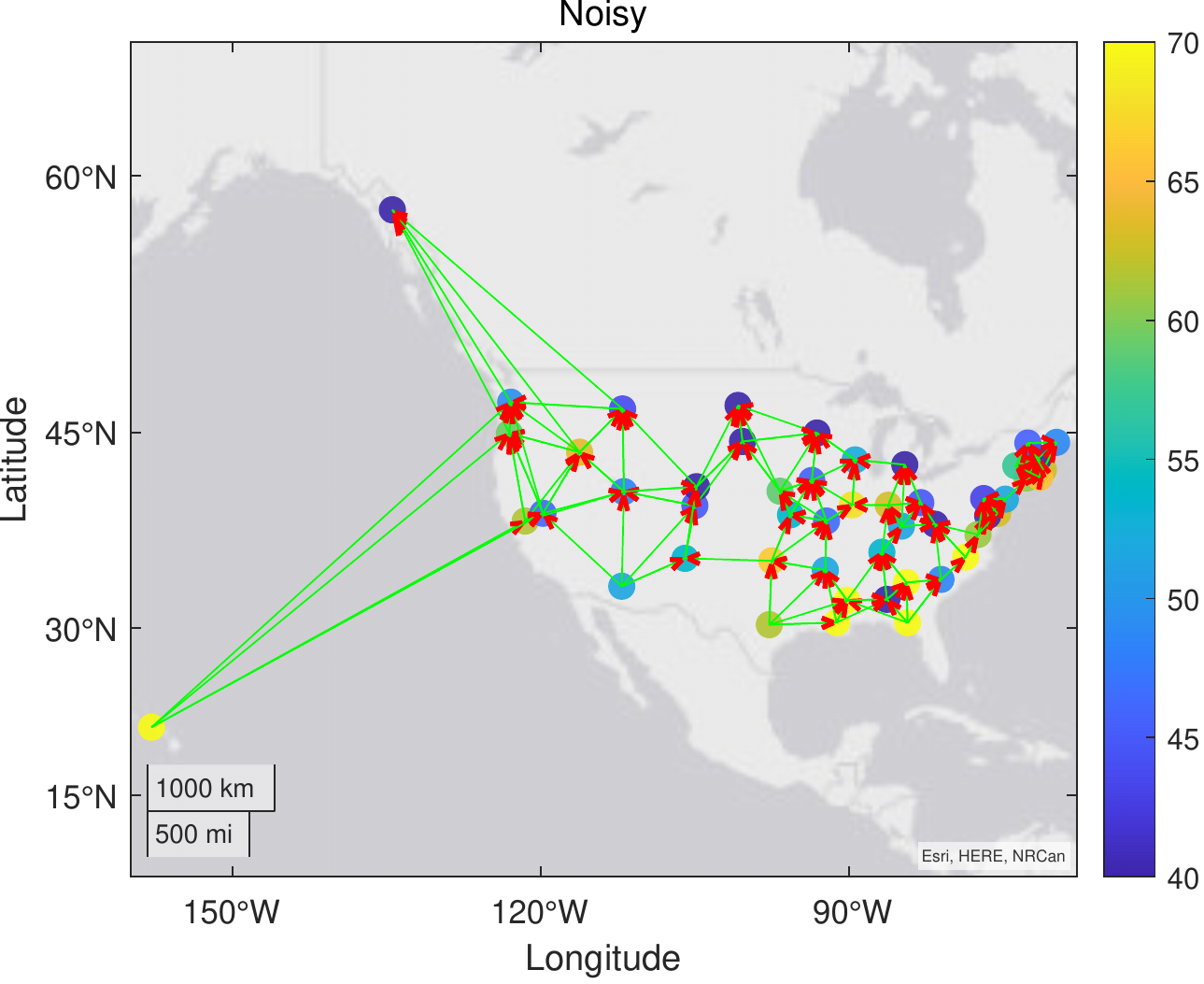}}
	\subfigure[denoised signal using Hermitian GFT \cite{furutani2019graph}.]{	  
		\label{direct_denoise_HL}
		\includegraphics[width=0.4\textwidth]{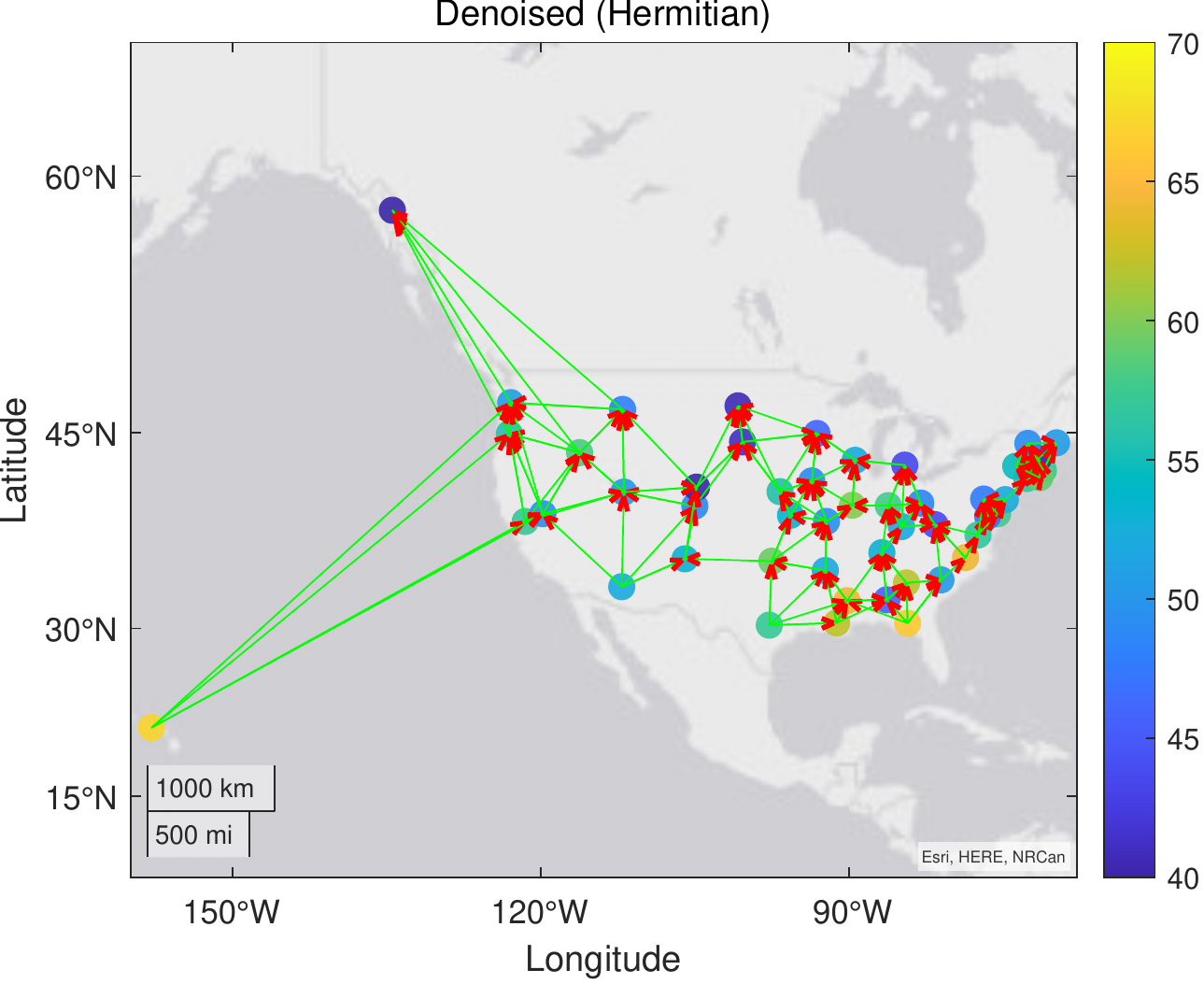}}
	\subfigure[denoised signal using DGFRFT.]{	  
		\label{direct_denoise_FrHL}
		\includegraphics[width=0.4\textwidth]{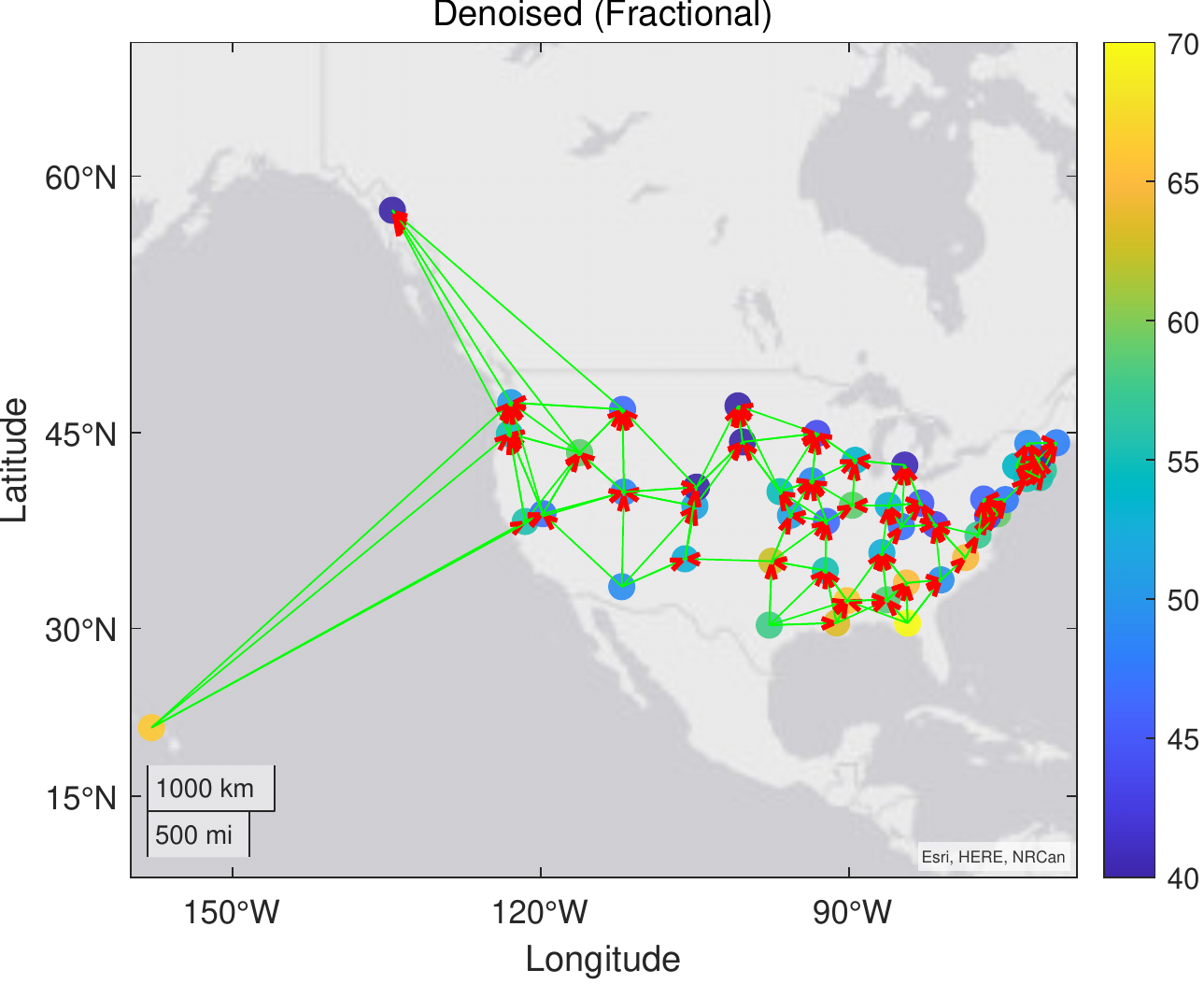}}
	\caption{Original, noisy and denoised temperature signals on directed US graph.}
	\label{directed}
\end{figure}

\begin{figure}[t]
	\centering
	\subfigure[original signal.]{
		\label{undirect_original}
		\includegraphics[width=0.3\textwidth]{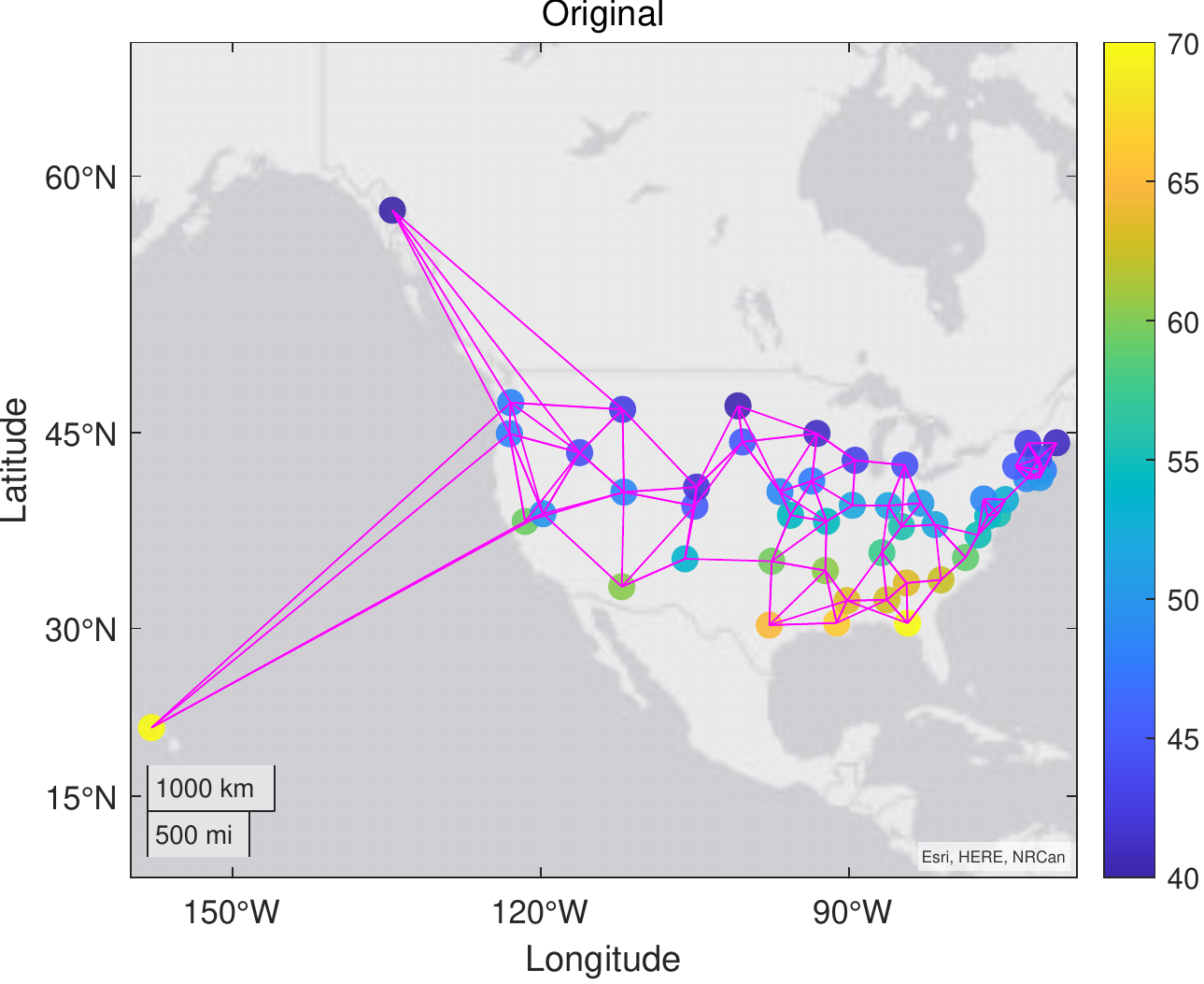}}
	\subfigure[noisy signal.]{
		\label{undirect_noisy}
		\includegraphics[width=0.3\textwidth]{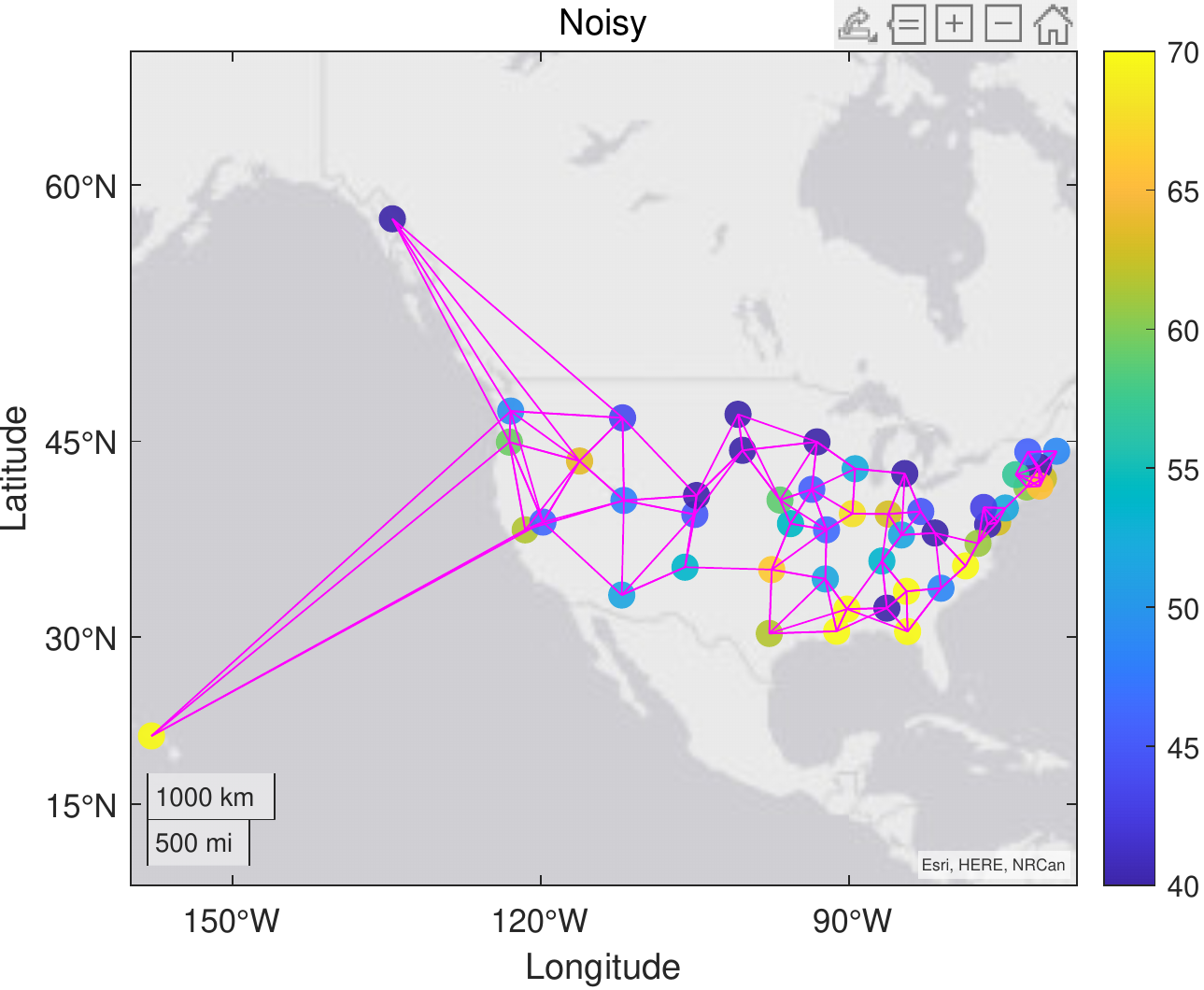}}
	\subfigure[denoised signal.]{	  
		\label{undirect_denoise_FrHL}
		\includegraphics[width=0.3\textwidth]{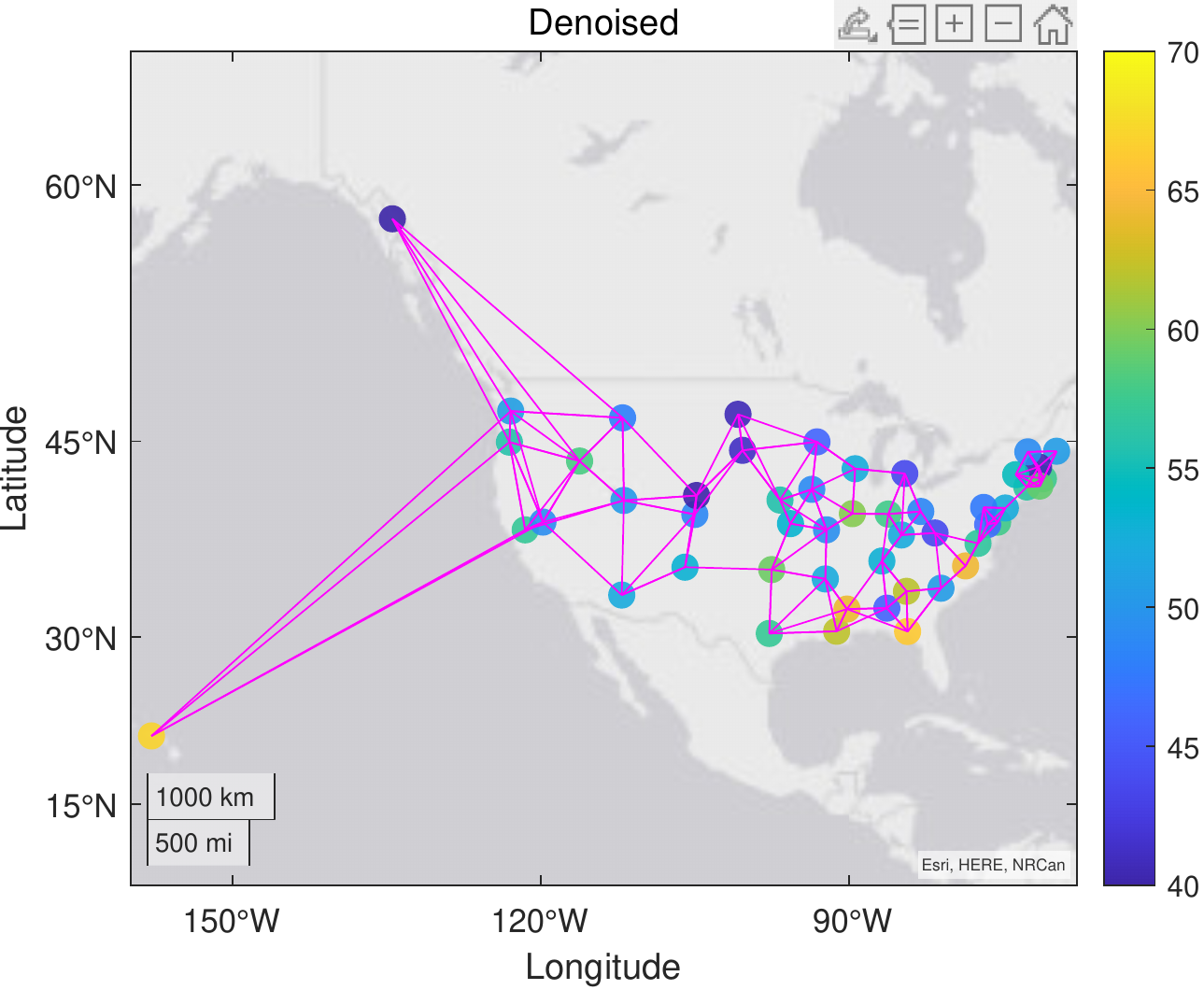}}
	\caption{Original, noisy and denoised temperature signals on undirected US graph.}
	\label{undirected}
\end{figure}

Here we assess the performance of DGFRFT via simulations on two graphs. In Simulation 1, we use the US temperature data to test that  directionality is an important key in application. For a graph that starts out as a directed graph, we apply DGFRFT to denoise. Then ignore the directionality of the graph, do the same denoising step, and observe the results. In graph fractional domain, DGFRFT and GFRFT \cite{25} are suitable for directed graph. The former is based on Laplacian matrix, and the latter uses adjacency matrix. In Simulation 2, for a real dataset, we use DGFRFT and GFRFT to perform the same filtering and denoising task, and then compare the results. The experimental results show that our DGFRFT introduces smaller errors and has higher robustness.

\subsection{Simulation 1: \textbf{Directed} VS \textbf{Undirected}}
In this section, real temperature data measured by 50 state meteorological stations in the United States (US) are used to make experiment. We consider signal denoising on a directed graph. First, based on a map of the US, we define a digraph using geographic locations and latitude. In this graph, the vertexes represent the 50 states of the US. The directed edges between states are based on latitudinal assignments from states with low latitude to states with high latitude. Only when two states share a border are they connected by an edge. The average annual temperature in each state is viewed as a graph signal. States at lower latitudes have higher average temperatures. Therefore, there is a correlation between latitude and temperature, and it is reasonable to use latitude to define a digraph. Fig. \ref{USTEM} shows the US temperature signal on a directed graph. The data is available in \sloppy\url{https://www.currentresults.com/Weather/US/average-annual-state-temperatures.php}.

We generate a noisy signal $g= f + n$, where $f$ is the original signals representing average temprature on the graph and $n$ is the noise vector whose coordinates are independently sampled from the Gaussian distribution with zero mean and standard deviation $\sigma = 10$. We apply GFT \cite{4}, Hemitian GFT \cite{furutani2019graph}, DGFRFT to recover the original signal from the noisy signal. For each method, we use a low-pass filter kernel
\begin{equation*}
	\hat{h_d}(\lambda) = \frac{1}{1+c\lambda},
\end{equation*}
where $\lambda$ is the eigenvalue corresponding to graph Laplacian, Hermitian graph Laplacian, graph Hermitian fractional Laplacian respectively. The denoised signal can be calculated as 
\begin{equation*}
	\tilde{f} = U\hat{H}U^*g,
\end{equation*}
where $\hat{H} = \text{diag}\{\hat{h_d}(\lambda_0),\ldots,\hat{h_d}(\lambda_{N-1})\}$ and $U$ is the GFT, Hermitian GFT, DGFRFT transform matrix. In this experiment, we set $c=0.02$, $q=0.5$ and $\alpha=0.9$. Note that we can only apply GFT \cite{4} to undirected graphs. To verify the importance of the directionality, we use the above denoising scheme using the same US graph in Fig. \ref{USTEM}, where we omit the direction of edges. Fig. \ref{directed} and \ref{undirected} give an intuitive example of the original, noisy, and denoised temperature signals on directed and undirected US graphs respectively. In the directed case, the performance difference of Hermitian GFT and DGFRFT is observed in Fig. \ref{direct_denoise_HL} and \ref{direct_denoise_FrHL}. In the undirected case, the denoising result after GFT is shown in Fig. \ref{undirect_denoise_FrHL}.

In addition, we calculate the root mean square error (RMSE) between original signal and the denoised signal obtained using GFT \cite{4}, Hermitian GFT \cite{furutani2019graph} and DGFRFT approaches. The results are described in Table \ref{tab:RE}. As the results demonstrate, these three transforms to graph signal denoising perform well, leading to small average errors. The RMSE obtained by our method is about 6.3828, while this quantity changes to 6.5685 and 6.5657 by using classical GFT based on undirected graph \cite{4} and another kind of GFT based on directed graph \cite{furutani2019graph} for denoising respectively. In conclusion, the proposed approach DGFRFT on directed graph outperforms other two ways, which highlights its practical usefulness in data denoising.

\begin{table*}[htbp]
	\centering
	\caption{Average RMSE between denoised signal and original signal.}
	\begin{tabular}{|c|c|c|}
		\hline
		Graph type & Transform method & RMSE\\
		\hline
		undirected & GFT \cite{4} & 6.5685\\
		directed & Hermitian GFT \cite{furutani2019graph} & 6.5657\\
		directed & DGFRFT & 6.3828\\
		\hline
	\end{tabular}
	\label{tab:RE}
\end{table*}

\subsection{Simulation 2: \textbf{DGFRFT} VS \textbf{GFRFT}}
Next, compared to GFRFT \cite{25}, which is also applicable to digraph, we study a real brain graph to demonstrate the superiority of DGFRFT in denoising tasks. The datasets represent the macaque large-scale visual and sensorimotor area corticocortical connectivity \cite{rubinov2010complex}. It has 47 vertexes and 505 edges (121 edges are directed). The vertexes represent cortical areas and edges represent large corticocortical tracts or functional associations. The data is available in \sloppy\url{https://sites.google.com/site/bctnet/}.

\begin{figure}[t]
	\centering
	\includegraphics[width=0.7\textwidth]{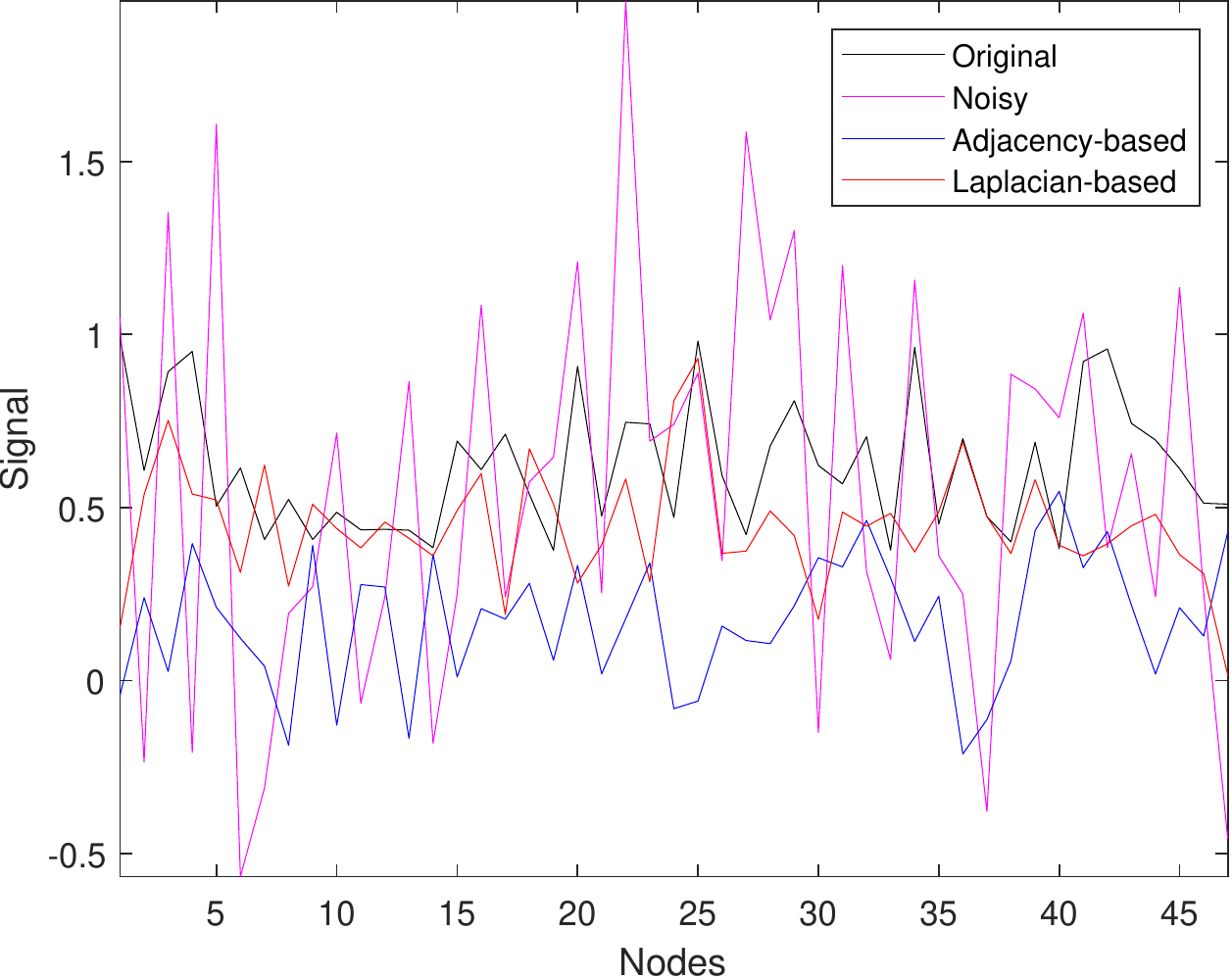}
	\caption{A sample realization of the original, noisy, and recovered signals of DGFRFT and GFRFT}
	\label{compared}
\end{figure}

\begin{figure}[t]
	\centering
	\subfigure[Relative recovery error of DGFRFT]{
		\label{laperror}
		\includegraphics[width=0.45\textwidth]{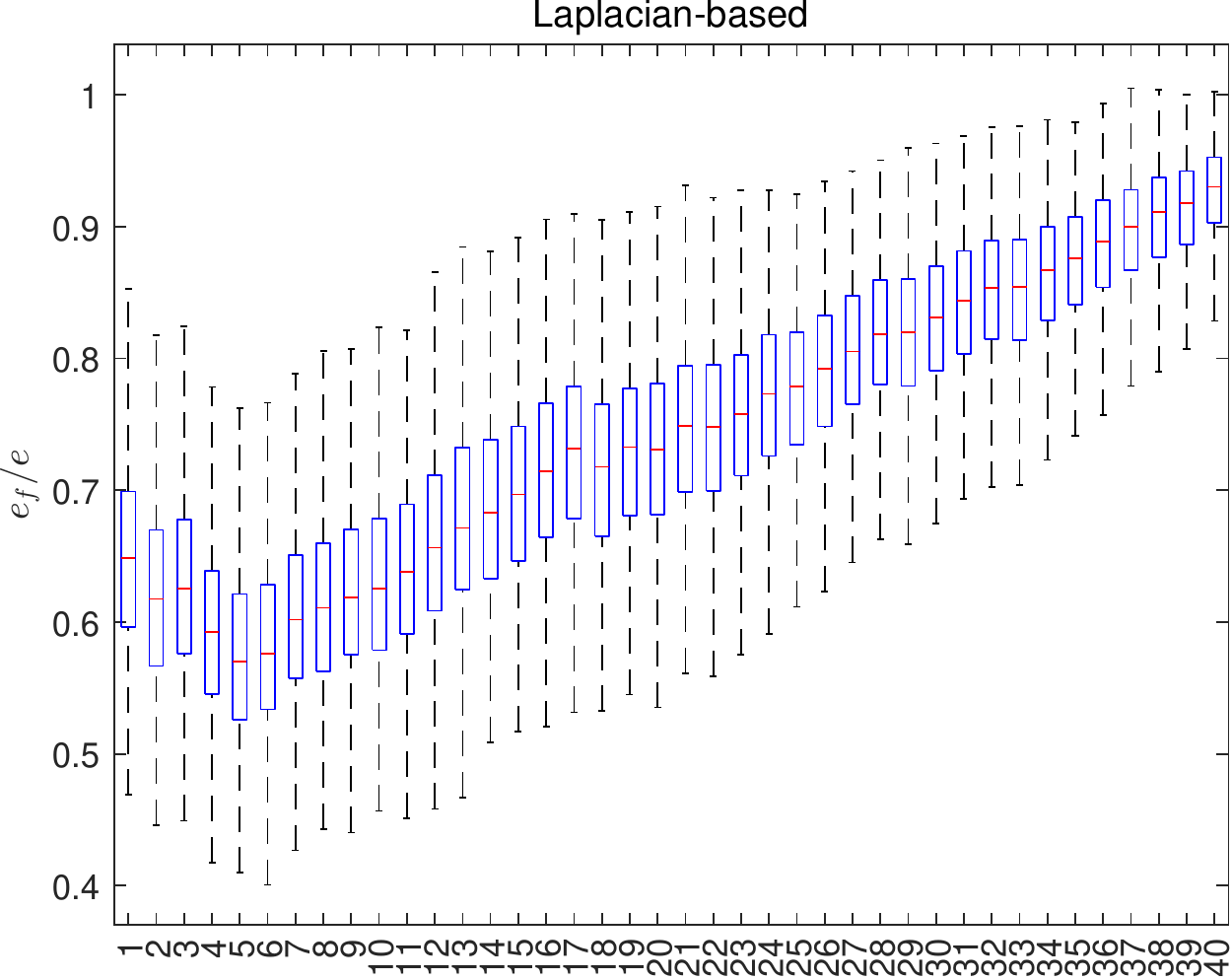}}
	\subfigure[Relative recovery error of GFRFT]{
		\label{adjerror}
		\includegraphics[width=0.45\textwidth]{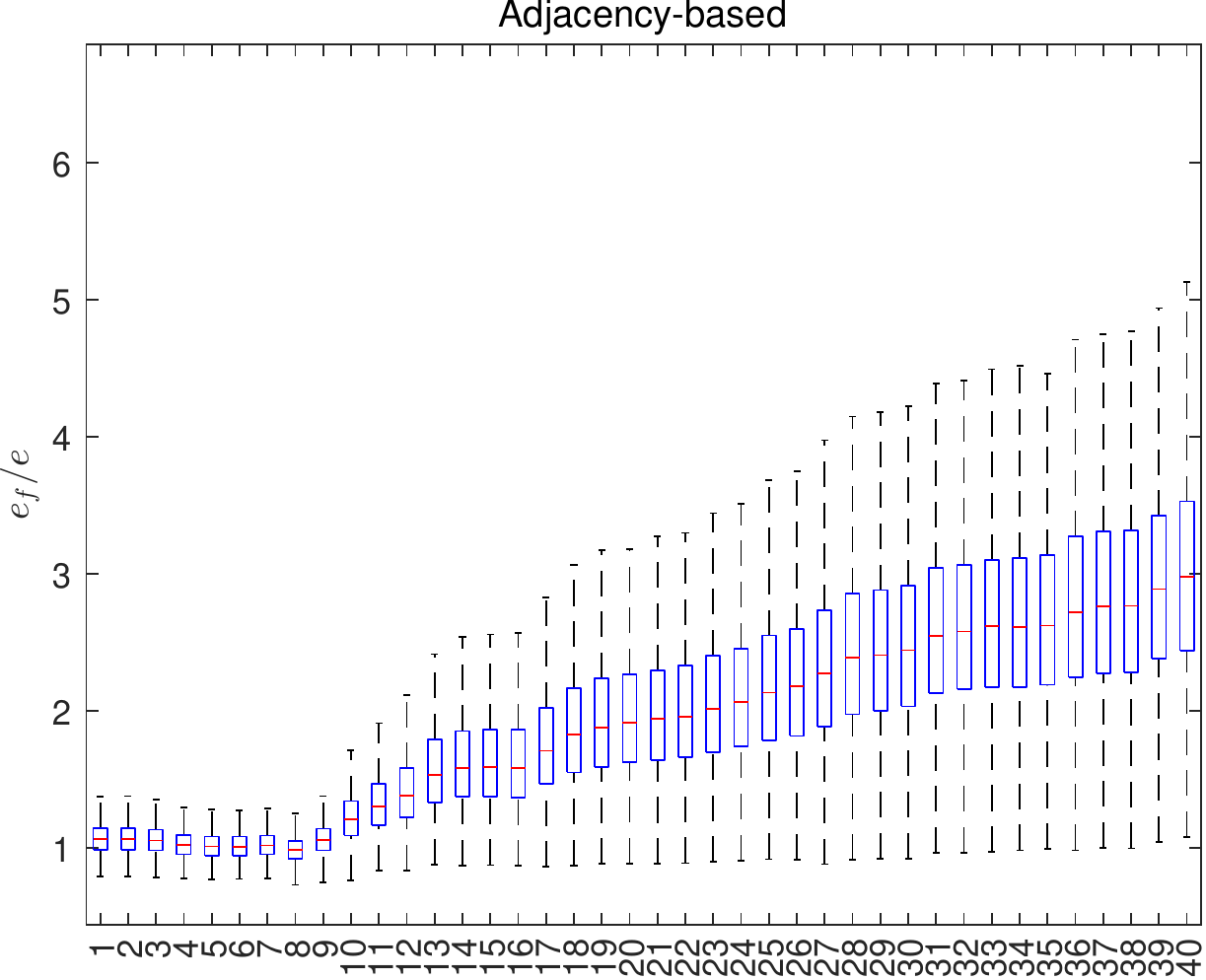}}
	\caption{Relative recovery error 
		of DGFRFT and GFRFT when $\alpha=0.9$ and $q=0.2$.}
	\label{error}
\end{figure}

Let $U$ be the orthonormal DGFRFT or GFRFT basis. We construct a synthetic graph signal by $x_f=e^{-f}$. We add Gaussian noise to the original signal to obtain noisy signal $g=x_f+n$. We use a series of filter kernels constructed by $H := \text{diag}(h)$, where $h_i = 1[i \le \ell]$ and $\ell$ is a tuning parameter to control the number of diagonal elements. The parameter $\ell$ can be viewed as spectral window size and when $\ell=47$ the filter kernel is just the identity matrix. The recovered signal is given by 
\begin{equation*}
	\tilde{x}_f = UHU^*g.
\end{equation*}
Fig. \ref{compared} shows an example of original, noisy, and recovered signals of DGFRFT and GFRFT with window size $5$. Moreover, we compute the relative recovery error with respect to the true error. The recovery error is defined as $e_f = \|\tilde{x}_f - x_f\|/\|x_f\|$. The true error is defined as $e = \|n\|/\|x\|$, and the relative recovery error is defined as $e_f/e$. Fig. \ref{error} are boxplots depicting $e_f/e$ versus $\ell$ averaged over 1000 Monte-Carlo simulations, and demonstrates the effectiveness of adopting filters along with the proposed two methods. We can see DGFRFT is much more stable compared with GFRFT. The Jordan decomposition required in GFRFT is numerically unstable, which may be responsible for the large reconstruction errors. Therefore, compared with GFRFT, DGFRFT has certain advantages when dealing with graph signals whose underlying structure is directed graph.

\section{Conclusions}
Signals defined on directed graphs have important practical significance. This paper proposes a methodology to carry out graph signals processing on directed graph in spectral graph fractional domain. First, we introduce a method to construct a new fractional Laplacian matrix for directed graph and prove that it is a positive semi-definite Hermitian matrix. Equipped with this Hermitian  fractional Laplacian matrix, a new transform named DGFRFT is presented. Then, to highlight the utility of DGFRFT, we propose two basic filtering method and use it to signal denoising on real data. Finally, the effectiveness of the DGFRFT construction is illustrated through tests on directed real world graphs. 

\bibliography{references}

\end{document}